\documentclass[12pt,twoside]{amsart}
\usepackage[latin1]{inputenc}
\usepackage{amsmath, amsthm, amscd, amsfonts, amssymb, graphicx}
\usepackage[bookmarksnumbered, plainpages]{hyperref}

\newtheorem{thm}{Theorem}[section]

\newtheorem{lem}[thm]{Lemma}

\newtheorem{defn}[thm]{Definition}

\numberwithin{equation}{section}

\begin{document}

\title{\bf Anomaly cancellation and modularity $: E_{8}\times E_{8}\times E_{8}$ case }
\author{Siyao Liu \hskip 0.4 true cm Yong Wang$^{*}$ \hskip 0.4 true cm Yuchen Yang }

\thanks{{\scriptsize
\hskip -0.4 true cm \textit{2010 Mathematics Subject Classification:}
58C20; 57R20; 53C80.
\newline \textit{Key words and phrases:} Modular invariance; Cancellation formulas; $E_{8}$ bundles.
\newline \textit{$^{*}$Corresponding author}}}

\maketitle

\begin{abstract}
 \indent In \cite{HHLZ} and \cite{WG}, the authors gave anomaly cancellation formulas for the gauge groups $E_{8},$ $E_{8}\times E_{8}.$ 
 In this paper, we mainly deal with the case of gauge group $E_{8}\times E_{8}\times E_{8}.$
 Using the $E_{8}\times E_{8}\times E_{8}$ bundle, we construct some modular forms over $SL_{2}(\mathbf{Z}).$
 By these modular forms, we get some new anomaly cancellation formulas of characteristic forms.
\end{abstract}

\vskip 0.2 true cm

%------------------------------------------------------------------------------------%

\pagestyle{myheadings}
\markboth{\rightline {\scriptsize Liu}}
         {\leftline{\scriptsize Anomaly cancellation and modularity $: E_{8}\times E_{8}\times E_{8}$ case}}

\bigskip
\bigskip

%------------------------------------------------------------------------------------%
%------------------------------------------------------------------------------------%

\section{ Introduction }

In \cite{AW}, Alvarez-Gaum\'{e} and Witten discovered the ``miraculous cancellation" formula.
This formula reveals a beautiful relation between the top components of the Hirzebruch $\widehat{L}$-form and $\widehat{A}$-form of a $12$-dimensional smooth Riemannian manifold $M.$
Liu established higher-dimensional ``miraculous cancellation" formulas for $(8k+4)$-dimensional Riemannian manifolds by developing modular invariance properties of characteristic forms \cite{L}.
Han and Zhang established a general cancellation formula that involves a complex line bundle for $(8k+4)$-dimensional smooth Riemannian manifold in \cite{HZ2,HZ}.
For higher-dimensional smooth Riemannian manifolds the authors obtained some cancellation formulas in \cite{HH,HLZ3,HLZ2,W1}.
And in \cite{LW}, the authors generalized the Han-Liu-Zhang cancellation formulas to the $(a, b)$ type cancellation formulas.

Wang and Yang twisted the Chen-Han-Zhang $SL_{2}(\mathbf{Z})$ modular form \cite{CHZ} by $E_{8}$ bundles and got $SL_{2}(\mathbf{Z})$ modular forms of weight $14$ and $10$ for $14$ and $10$ dimensional spin manifolds \cite{WY}.
They also twisted the Liu's modular form \cite{L2} by $E_{8}$ bundles and got $\Gamma^{0}(2)$ and $\Gamma_{0}(2)$ modular forms of weight $14$ and $10$ for $12$ dimensional spin manifold.
In \cite{HLZ}, Han, Liu and Zhang showed that both of the Ho$\check{r}$ava-Witten anomaly factorization formula for the gauge group $E_{8}$ and the Green-Schwarz anomaly factorization formula for the gauge group $E_{8}\times E_{8}$ could be derived through modular forms of weight $14.$
In \cite{HHLZ}, Han, Huang, Liu and Zhang introduced a modular form of weight $14$ over $SL_{2}(\mathbf{Z})$ and a modular form of weight $10$ over $SL_{2}(\mathbf{Z})$ by $E_{8}$ bundles and they got some interesting anomaly cancellation formulas on $12$ dimensional manifolds.

The motivation of this paper is to prove more anomaly cancellation formulas by twisting $E_{8}\times E_{8}\times E_{8}$ bundles.
We get some $SL_{2}(\mathbf{Z})$ modular forms of weight $10,$ $14,$ $18$ by twisting $E_{8}\times E_{8}\times E_{8}$ bundles and we get some new anomaly cancellation formulas for $12$ dimensional oriented smooth closed manifold.
We also obtain anomaly cancellation formulas for the gauge groups $E_{8}\times E_{8}\times E_{8}.$

A brief description of the organization of this paper is as follows.
In Section 2, we give some definitions and basic notions that we will use in this paper.
Section 3 contains a discussion of anomaly cancellation formulas for the gauge groups $E_{8},$ $E_{8}\times E_{8}$ and $E_{8}\times E_{8}\times E_{8}.$
The summary of the paper is given in Section 4.

%------------------------------------------------------------------------------------%

\vskip 1 true cm

\section{ The basic knowledge }

Firstly, we give some definitions and basic notions that will be used throughout the paper.
For the details, see \cite{A,H,HW,K,Z}.

Let $Z$ be a 12 dimensional oriented smooth closed manifold, $TZ$ be the tangent bundle of $Z,$ $\nabla^{TZ}$ be the associated Levi-Civita connection on $TZ$ and $R^{TZ}=(\nabla^{TZ})^{2}$ be the curvature of $\nabla^{TZ}.$
According to the detailed descriptions in \cite{Z}, let $\widehat{A}(TZ)$ and $\widehat{L}(TZ)$ be the Hirzebruch characteristic forms defined respectively by
\begin{align}
\widehat{A}(TZ)&=\det\nolimits^{\frac{1}{2}}\left(\frac{\frac{\sqrt{-1}}{4\pi}R^{TZ}}{\sinh(\frac{\sqrt{-1}}{4\pi}R^{TZ})}\right),\\
\widehat{L}(TZ)&=\det\nolimits^{\frac{1}{2}}\left(\frac{\frac{\sqrt{-1}}{2\pi}R^{TZ}}{\tanh(\frac{\sqrt{-1}}{4\pi}R^{TZ})}\right).
\end{align}

Let $E, F$ be two Hermitian vector bundles over $Z$ carrying Hermitian connection $\nabla^{E}, \nabla^{F}$ respectively.
Let $R^{E}=(\nabla^{E})^{2}$ (resp. $R^{F}=(\nabla^{F})^{2}$) be the curvature of $\nabla^{E}$ (resp. $\nabla^{F}$).
If we set the formal difference $G=E-F,$ then $G$ carries an induced Hermitian connection $\nabla^{G}$ in an obvious sense.
We define the associated Chern character form as
\begin{align}
\text{ch}(G)=\text{tr}\left[\exp\left(\frac{\sqrt{-1}}{2\pi}R^{E}\right)\right]-\text{tr}\left[\exp\left(\frac{\sqrt{-1}}{2\pi}R^{F}\right)\right].
\end{align}

For any complex number $q,$ let
\begin{align}
\wedge_{q}(E)&=\mathbf{C}|_{M}+qE+q^{2}\wedge^{2}(E)+\cdot\cdot\cdot,\\
S_{q}(E)&=\mathbf{C}|_{M}+qE+q^{2}S^{2}(E)+\cdot\cdot\cdot,
\end{align}
denote respectively the total exterior and symmetric powers of $E,$ which live in $K(Z)[[q]].$
The following relations between these operations hold
\begin{align}
S_{q}(E)=\frac{1}{\wedge_{-q}(E)},~~~~ \wedge_{q}(E-F)=\frac{\wedge_{q}(E)}{\wedge_{q}(F)}.
\end{align}

If $W$ is a real Euclidean vector bundle over $Z$ carrying a Euclidean connection $\nabla^{W},$ then its complexification $W_{\mathbf{C}}=W\otimes \mathbf{C}$ is a complex vector bundle over $Z$ carrying a canonical induced Hermitian metric from that of $W,$ as well as a Hermitian connection $\nabla^{W_{\mathbf{C}}}$ induced from $\nabla^{W}.$
If $E$ is a complex vector bundle over $Z,$ set $\widetilde{E}=E-\dim E.$
If $\omega$ is a differential form over $Z,$ we denote by $\omega^{(i)}$ its degree $i$ component.

Refer to \cite{C}, we recall the four Jacobi theta functions are defined as follows:
\begin{align}
&\theta(v, \tau)=2q^{\frac{1}{8}}\sin(\pi v)\prod^{\infty}_{j=1}[(1-q^{j})(1-e^{2\pi\sqrt{-1}v}q^{j})(1-e^{-2\pi\sqrt{-1}v}q^{j})];\\
&\theta_{1}(v, \tau)=2q^{\frac{1}{8}}\cos(\pi v)\prod^{\infty}_{j=1}[(1-q^{j})(1+e^{2\pi\sqrt{-1}v}q^{j})(1+e^{-2\pi\sqrt{-1}v}q^{j})];\\
&\theta_{2}(v, \tau)=\prod^{\infty}_{j=1}[(1-q^{j})(1-e^{2\pi\sqrt{-1}v}q^{j-\frac{1}{2}})(1-e^{-2\pi\sqrt{-1}v}q^{j-\frac{1}{2}})];\\
&\theta_{3}(v, \tau)=\prod^{\infty}_{j=1}[(1-q^{j})(1+e^{2\pi\sqrt{-1}v}q^{j-\frac{1}{2}})(1+e^{-2\pi\sqrt{-1}v}q^{j-\frac{1}{2}})],
\end{align}
where $q=e^{2\pi\sqrt{-1}\tau}$ with $\tau\in\mathbf{H},$ the upper half complex plane.

Let
\begin{align}
\theta'(0, \tau)=\frac{\partial \theta(v, \tau)}{\partial v}\Big|_{v=0}.
\end{align}
Then the following Jacobi identity holds,
\begin{align}
\theta'(0, \tau)=\pi\theta_{1}(0, \tau)\theta_{2}(0, \tau)\theta_{3}(0, \tau).
\end{align}

In what follows,
\begin{align}
SL_{2}(\mathbf{Z})=\Big\{ \Big( \begin{array}{cc}
a & b\\
c & d
\end{array}
\Big)\Big|a, b, c, d\in\mathbf{Z}, ad-bc=1
\Big \}
\end{align}
stands for the modular group.
Write $S=\Big(\begin{array}{cc}
0 & -1\\
1 & 0
\end{array}\Big),
T=\Big(\begin{array}{cc}
1 & 1\\
0 & 1
\end{array}\Big)$
be the two generators of $SL_{2}(\mathbf{Z}).$
They act on $\mathbf{H}$ by $S\tau=-\frac{1}{\tau}, T\tau=\tau+1.$
One has the following transformation laws of theta functions under the actions of $S$ and $T:$
\begin{align}
&\theta(v, \tau+1)=e^{\frac{\pi\sqrt{-1}}{4}}\theta(v, \tau),~~~~ \theta(v, -\frac{1}{\tau})=\frac{1}{\sqrt{-1}}\Big(\frac{\tau}{\sqrt{-1}}\Big)^{\frac{1}{2}}e^{\pi\sqrt{-1}\tau v^{2}}\theta(\tau v, \tau);\\
&\theta_{1}(v, \tau+1)=e^{\frac{\pi\sqrt{-1}}{4}}\theta_{1}(v, \tau),~~~~ \theta_{1}(v, -\frac{1}{\tau})=\Big(\frac{\tau}{\sqrt{-1}}\Big)^{\frac{1}{2}}e^{\pi\sqrt{-1}\tau v^{2}}\theta_{2}(\tau v, \tau);\\
&\theta_{2}(v, \tau+1)=\theta_{3}(v, \tau),~~~~ \theta_{2}(v, -\frac{1}{\tau})=\Big(\frac{\tau}{\sqrt{-1}}\Big)^{\frac{1}{2}}e^{\pi\sqrt{-1}\tau v^{2}}\theta_{1}(\tau v, \tau);\\
&\theta_{3}(v, \tau+1)=\theta_{2}(v, \tau),~~~~ \theta_{3}(v, -\frac{1}{\tau})=\Big(\frac{\tau}{\sqrt{-1}}\Big)^{\frac{1}{2}}e^{\pi\sqrt{-1}\tau v^{2}}\theta_{3}(\tau v, \tau).
\end{align}
Differentiating the above transformation formulas, we get that
\begin{align}
&\theta'(v, \tau+1)=e^{\frac{\pi\sqrt{-1}}{4}}\theta'(v, \tau),\\
&\theta'(v, -\frac{1}{\tau})=\frac{1}{\sqrt{-1}}\Big(\frac{\tau}{\sqrt{-1}}\Big)^{\frac{1}{2}}e^{\pi\sqrt{-1}\tau v^{2}}(2\pi\sqrt{-1}\tau v\theta(\tau v, \tau)+\tau\theta'(\tau v, \tau));\nonumber\\
&\theta_{1}'(v, \tau+1)=e^{\frac{\pi\sqrt{-1}}{4}}\theta_{1}'(v, \tau),\\
&\theta_{1}'(v, -\frac{1}{\tau})=\Big(\frac{\tau}{\sqrt{-1}}\Big)^{\frac{1}{2}}e^{\pi\sqrt{-1}\tau v^{2}}(2\pi\sqrt{-1}\tau v\theta_{2}(\tau v, \tau)+\tau\theta'_{2}(\tau v, \tau));\nonumber\\
&\theta'_{2}(v, \tau+1)=\theta'_{3}(v, \tau),\\
&\theta'_{2}(v, -\frac{1}{\tau})=\Big(\frac{\tau}{\sqrt{-1}}\Big)^{\frac{1}{2}}e^{\pi\sqrt{-1}\tau v^{2}}(2\pi\sqrt{-1}\tau v\theta_{1}(\tau v, \tau)+\tau\theta'_{1}(\tau v, \tau));\nonumber\\
&\theta'_{3}(v, \tau+1)=\theta'_{2}(v, \tau),\\
&\theta'_{3}(v, -\frac{1}{\tau})=\Big(\frac{\tau}{\sqrt{-1}}\Big)^{\frac{1}{2}}e^{\pi\sqrt{-1}\tau v^{2}}(2\pi\sqrt{-1}\tau v\theta_{3}(\tau v, \tau)+\tau\theta'_{3}(\tau v, \tau)).\nonumber
\end{align}
Therefore
\begin{align}
\theta'(0, -\frac{1}{\tau})=\frac{1}{\sqrt{-1}}\Big(\frac{\tau}{\sqrt{-1}}\Big)^{\frac{1}{2}}\tau\theta'(0, \tau).
\end{align}

\begin{defn}
A modular form over $\Gamma,$ a subgroup of $SL_{2}(\mathbf{Z}),$ is a holomorphic function $f(\tau)$ on $\mathbf{H}$ such that
\begin{align}
f(g\tau):=f\Big(\frac{a\tau+b}{c\tau+d}\Big)=\chi(g)(c\tau+d)^{k}f(\tau),~\forall g=\Big(\begin{array}{cc}
a & b\\
c & d
\end{array}\Big)\in\Gamma,
\end{align}
where $\chi:\Gamma\rightarrow\mathbf{C}^{*}$ is a character of $\Gamma,$ $k$ is called the weight of $f.$
\end{defn}

Let $E_{2}(\tau)$ be Eisenstein series which is a quasimodular form over $SL_{2}(\mathbf{Z}),$ satisfying
\begin{align}
E_{2}\Big(\frac{a\tau+b}{c\tau+d}\Big)=(c\tau+d)^{2}E_{2}(\tau)-\frac{6\sqrt{-1}c(c\tau+d)}{\pi}.
\end{align}
In particular, we have
\begin{align}
&E_{2}(\tau+1)=E_{2}(\tau),\\
&E_{2}\Big(-\frac{1}{\tau}\Big)=\tau^{2}E_{2}(\tau)-\frac{6\sqrt{-1}\tau}{\pi},
\end{align}
and
\begin{align}
E_{2}(\tau)=1-24q-72q^{2}+\cdot\cdot\cdot,
\end{align}
where the $``\cdot\cdot\cdot"$ terms are the higher degree terms, all of which have integral coefficients.

For the principal $E_{8}$ bundles $P_{i}, i=1,2$ consider the associated bundles
\begin{align}
\mathcal{V}_{i}=\sum_{k=0}^{\infty}(P_{i}\times_{\rho_{k}}V_{k})q^{k}\in K(Z)[[q]].
\end{align}
Let $W_{i}=P_{i}\times_{\rho_{1}}V_{1}, i=1,2$ be the complex vector bundles associated to the adjoint representation $\rho_{1},$ $\overline{W_{i}}=P_{i}\times_{\rho_{2}}V_{2}, i=1,2$ be the complex vector bundles associated to the adjoint representation $\rho_{2}.$

Following the notation of \cite{HLZ}, we have that there are formal two forms $y_{l}^{i}, 1\leq l\leq 8, i=1,2$ such that
\begin{align}
\varphi(\tau)^{(8)}ch(\mathcal{V}_{i})=\frac{1}{2}\Big(\prod_{l=1}^{8}\theta_{1}(y_{l}^{i}, \tau)+\prod_{l=1}^{8}\theta_{2}(y_{l}^{i}, \tau)+\prod_{l=1}^{8}\theta_{3}(y_{l}^{i}, \tau)\Big),
\end{align}
and
\begin{align}
\sum_{l=1}^{8}(2\pi\sqrt{-1}y_{l}^{i})^{2}=-\frac{1}{30}c_{2}(W_{i}),
\end{align}
where $\varphi(\tau)=\prod_{n=1}^{\infty}(1-q^{n}),$ $c_{2}(W_{i})$ denotes the second Chern class of $W_{i}.$

\section{ Anomaly cancellation formulas }

\subsection{ Anomaly cancellation formulas for the gauge group $E_{8}$ }

Let $\triangle(Z)$ be the spinor bundle.
Set
\begin{align}
\Theta_{1}(T_{\mathbf{C}}Z)=&\bigotimes^{\infty}_{n=1}S_{q^{n}}(\widetilde{T_{\mathbf{C}}Z})\otimes\bigotimes_{m=1}^{\infty}\wedge_{q^{m}}(\widetilde{T_{\mathbf{C}}Z})\in K(Z)[[q]],\\
\Theta_{2}(T_{\mathbf{C}}Z)=&\bigotimes^{\infty}_{n=1}S_{q^{n}}(\widetilde{T_{\mathbf{C}}Z})\otimes\bigotimes_{m=1}^{\infty}\wedge_{-q^{m-\frac{1}{2}}}(\widetilde{T_{\mathbf{C}}Z})\in K(Z)[[q^{\frac{1}{2}}]],\\
\Theta_{3}(T_{\mathbf{C}}Z)=&\bigotimes^{\infty}_{n=1}S_{q^{n}}(\widetilde{T_{\mathbf{C}}Z})\otimes\bigotimes_{m=1}^{\infty}\wedge_{q^{m-\frac{1}{2}}}(\widetilde{T_{\mathbf{C}}Z})\in K(Z)[[q^{\frac{1}{2}}]].
\end{align}
Consider
\begin{align}
Q_{1}(Z, \tau)=&\int_{Z}e^{\frac{1}{24}E_{2}(\tau)\cdot\frac{1}{30}c_{2}(W_{i})}\widehat{A}(TZ)[ch(\triangle(Z))ch(\Theta_{1}(T_{\mathbf{C}}Z))+2^{6}ch(\Theta_{2}(T_{\mathbf{C}}Z))\\
&+2^{6}ch(\Theta_{3}(T_{\mathbf{C}}Z))]\varphi(\tau)^{(8)}ch(\mathcal{V}_{i}).\nonumber
\end{align}

\begin{lem}
$Q_{1}(Z, \tau)$ is a modular form of weight $10$ over $SL_{2}(\mathbf{Z}).$
\end{lem}
\begin{proof}
Let $\{ \pm2\pi\sqrt{-1}x_{k}|1\leq k\leq 6\}$  be the formal Chern roots for $T_{\mathbf{C}}Z,$ then we have
\begin{align}
&\widehat{A}(TZ)ch(\triangle(Z))ch(\Theta_{1}(T_{\mathbf{C}}Z))=\prod_{k=1}^{6}\frac{2x_{k}\theta'(0, \tau)}{\theta(x_{k}, \tau)}\frac{\theta_{1}(x_{k}, \tau)}{\theta_{1}(0, \tau)},\\
&\widehat{A}(TZ)ch(2^{6}\Theta_{2}(T_{\mathbf{C}}Z))=\prod_{k=1}^{6}\frac{2x_{k}\theta'(0, \tau)}{\theta(x_{k}, \tau)}\frac{\theta_{2}(x_{k}, \tau)}{\theta_{2}(0, \tau)},\\
&\widehat{A}(TZ)ch(2^{6}\Theta_{3}(T_{\mathbf{C}}Z))=\prod_{k=1}^{6}\frac{2x_{k}\theta'(0, \tau)}{\theta(x_{k}, \tau)}\frac{\theta_{3}(x_{k}, \tau)}{\theta_{3}(0, \tau)}.
\end{align}
Accordingly,
\begin{align}
Q_{1}(Z, \tau)&=e^{\frac{1}{24}E_{2}(\tau)\cdot\frac{1}{30}c_{2}(W_{i})}\prod_{k=1}^{6}\frac{2x_{k}\theta'(0, \tau)}{\theta(x_{k}, \tau)}\Big(\prod_{k=1}^{6}\frac{\theta_{1}(x_{k}, \tau)}{\theta_{1}(0, \tau)}+\prod_{k=1}^{6}\frac{\theta_{2}(x_{k}, \tau)}{\theta_{2}(0, \tau)}+\prod_{k=1}^{6}\frac{\theta_{3}(x_{k}, \tau)}{\theta_{3}(0, \tau)}\Big)\\
&\cdot\frac{1}{2}\Big(\prod_{l=1}^{8}\theta_{1}(y_{l}^{i}, \tau)+\prod_{l=1}^{8}\theta_{2}(y_{l}^{i}, \tau)+\prod_{l=1}^{8}\theta_{3}(y_{l}^{i}, \tau)\Big).\nonumber
\end{align}
Moreover, the following identity holds
\begin{align}
&Q_{1}(Z, \tau+1)=Q_{1}(Z, \tau),\\
&Q_{1}(Z, -\frac{1}{\tau})=\tau^{10}Q_{1}(Z, \tau).
\end{align}
From the above it follows that $Q_{1}(Z, \tau)$ is a modular form over $SL_{2}(\mathbf{Z})$ with the weight $10.$
\end{proof}

\begin{thm}
For $12$ dimensional oriented smooth closed manifold, we conclude that
\begin{align}
&\Big\{e^{\frac{1}{720}c_{2}(W_{i})}\widehat{A}(TZ)\frac{1}{30}c_{2}(W_{i})(ch(\triangle(Z))+128)\Big\}^{(12)}=
\{e^{\frac{1}{720}c_{2}(W_{i})}\widehat{A}(TZ)(ch(\triangle(Z))\\
&\cdot ch(256+2\widetilde{T_{\mathbf{C}}Z}+W_{i})+ch(32768+128\widetilde{T_{\mathbf{C}}Z}+128W_{i}+128\wedge^{2}(\widetilde{T_{\mathbf{C}}Z})))\}^{(12)}.\nonumber
\end{align}
\end{thm}
\begin{proof}
A direct computation shows that 
\begin{align}
\Theta_{1}(T_{\mathbf{C}}Z)=&1+2q\widetilde{T_{\mathbf{C}}Z}+\cdot\cdot\cdot,\\
\Theta_{2}(T_{\mathbf{C}}Z)=&1-q^{\frac{1}{2}}\widetilde{T_{\mathbf{C}}Z}+q(\widetilde{T_{\mathbf{C}}Z}+\wedge^{2}(\widetilde{T_{\mathbf{C}}Z}))+\cdot\cdot\cdot,\\
\Theta_{3}(T_{\mathbf{C}}Z)=&1+q^{\frac{1}{2}}\widetilde{T_{\mathbf{C}}Z}+q(\widetilde{T_{\mathbf{C}}Z}+\wedge^{2}(\widetilde{T_{\mathbf{C}}Z}))+\cdot\cdot\cdot.
\end{align}
This clearly forces
\begin{align}
&e^{\frac{1}{24}E_{2}(\tau)\cdot\frac{1}{30}c_{2}(W_{i})}\widehat{A}(TZ)[ch(\triangle(Z))ch(\Theta_{1}(T_{\mathbf{C}}Z))+2^{6}ch(\Theta_{2}(T_{\mathbf{C}}Z))+2^{6}ch(\Theta_{3}(T_{\mathbf{C}}Z))]\\
&\cdot\varphi(\tau)^{(8)}ch(\mathcal{V}_{i})=e^{\frac{1}{720}c_{2}(W_{i})}\widehat{A}(TZ)(ch(\triangle(Z))+128)+q\Big(e^{\frac{1}{720}c_{2}(W_{i})}\widehat{A}(TZ)\big(ch(\triangle(Z))\nonumber\\
&\cdot ch(-8+2\widetilde{T_{\mathbf{C}}Z}+W_{i})+ch(-1024+128\widetilde{T_{\mathbf{C}}Z}+128W_{i}+128\wedge^{2}(\widetilde{T_{\mathbf{C}}Z}))-\frac{1}{30}c_{2}(W_{i})\nonumber\\
&\cdot(ch(\triangle(Z))+128)\big)\Big)+\cdot\cdot\cdot.\nonumber
\end{align}

It is well known that modular forms over $SL_{2}(\mathbf{Z})$ can be expressed as polynomials of the Einsentein series $E_{4}(\tau)$ and $E_{6}(\tau),$ where
\begin{align}
&E_{4}(\tau)=1+240q+2160q^{2}+6720q^{3}+\cdot\cdot\cdot,\\
&E_{6}(\tau)=1-504q-16632q^{2}-122976q^{3}+\cdot\cdot\cdot.
\end{align}
Their weights are $4$ and $6$ respectively.

Since $Q_{1}(Z, \tau)$ is a modular form over $SL_{2}(\mathbf{Z})$ with the weight $10,$ it must be a multiple of
\begin{align}
E_{4}(\tau)E_{6}(\tau)=1-264q+\cdot\cdot\cdot.
\end{align}

By comparing the constant term and the coefficients of $q,$ we can get Theorem 3.2.
\end{proof}

\begin{lem}
Suppose that $c_{2}(W_{i})=0,$ we have
\begin{align}
&-\{\widehat{A}(TZ)ch(\triangle(Z))ch(256+2\widetilde{T_{\mathbf{C}}Z}+W_{i})\}^{(12)}\\
&=\{\widehat{A}(TZ)ch(32768+128\widetilde{T_{\mathbf{C}}Z}+128W_{i}+128\wedge^{2}(\widetilde{T_{\mathbf{C}}Z}))) \}^{(12)}.\nonumber
\end{align}
\end{lem}

\subsection{ Anomaly cancellation formulas for the gauge group $E_{8}\times E_{8}$ }

Set
\begin{align}
Q_{2}(Z, \tau)=&\int_{Z}e^{\frac{1}{24}E_{2}(\tau)\cdot\frac{1}{30}(c_{2}(W_{i})+c_{2}(W_{j}))}\widehat{A}(TZ)[ch(\triangle(Z))ch(\Theta_{1}(T_{\mathbf{C}}Z))+2^{6}ch(\Theta_{2}(T_{\mathbf{C}}Z))\\
&+2^{6}ch(\Theta_{3}(T_{\mathbf{C}}Z))]\varphi(\tau)^{(16)}ch(\mathcal{V}_{i})ch(\mathcal{V}_{j}).\nonumber
\end{align}

Then we have
\begin{lem}
$Q_{2}(Z, \tau)$ is a modular form of weight $14$ over $SL_{2}(\mathbf{Z}).$
\end{lem}
\begin{proof}
By applying the formula shown in (3.5)-(3.7), we can calculate
\begin{align}
&Q_{2}(Z, \tau)=\\
&e^{\frac{1}{24}E_{2}(\tau)\cdot\frac{1}{30}(c_{2}(W_{i})+c_{2}(W_{j}))}\prod_{k=1}^{6}\frac{2x_{k}\theta'(0, \tau)}{\theta(x_{k}, \tau)}\Big(\prod_{k=1}^{6}\frac{\theta_{1}(x_{k}, \tau)}{\theta_{1}(0, \tau)}+\prod_{k=1}^{6}\frac{\theta_{2}(x_{k}, \tau)}{\theta_{2}(0, \tau)}+\prod_{k=1}^{6}\frac{\theta_{3}(x_{k}, \tau)}{\theta_{3}(0, \tau)}\Big)\nonumber\\
&\cdot\frac{1}{4}\Big(\prod_{l=1}^{8}\theta_{1}(y_{l}^{i}, \tau)+\prod_{l=1}^{8}\theta_{2}(y_{l}^{i}, \tau)+\prod_{l=1}^{8}\theta_{3}(y_{l}^{i}, \tau)\Big)\Big(\prod_{l=1}^{8}\theta_{1}(y_{l}^{j}, \tau)+\prod_{l=1}^{8}\theta_{2}(y_{l}^{j}, \tau)+\prod_{l=1}^{8}\theta_{3}(y_{l}^{j}, \tau)\Big).\nonumber
\end{align}
On the basis of (2.25) and (2.26), we have 
\begin{align}
&Q_{2}(Z, \tau+1)=Q_{2}(Z, \tau),\\
&Q_{2}(Z, -\frac{1}{\tau})=\tau^{14}Q_{2}(Z, \tau).
\end{align}
Hence $Q_{2}(Z, \tau)$ is a modular form over $SL_{2}(\mathbf{Z})$ with the weight $14.$
\end{proof}

\begin{thm}
For $12$ dimensional oriented smooth closed manifold, we can assert that
\begin{align}
&\Big\{e^{\frac{1}{720}(c_{2}(W_{i})+c_{2}(W_{j}))}\widehat{A}(TZ)\frac{1}{30}(c_{2}(W_{i})+c_{2}(W_{j}))(ch(\triangle(Z))+128)\Big\}^{(12)}\\
&=\{e^{\frac{1}{720}(c_{2}(W_{i})+c_{2}(W_{j}))}\widehat{A}(TZ)(ch(\triangle(Z))ch(8+2\widetilde{T_{\mathbf{C}}Z}+W_{i}+W_{j})\nonumber\\
&+ch(1024+128\widetilde{T_{\mathbf{C}}Z}+128W_{i}+128W_{j}+128\wedge^{2}(\widetilde{T_{\mathbf{C}}Z}))) \}^{(12)}.\nonumber
\end{align}
\end{thm}
\begin{proof}
Via simple calculations, we can obtain
\begin{align}
&e^{\frac{1}{24}E_{2}(\tau)\cdot\frac{1}{30}(c_{2}(W_{i})+c_{2}(W_{j}))}\widehat{A}(TZ)[ch(\triangle(Z))ch(\Theta_{1}(T_{\mathbf{C}}Z))+2^{6}ch(\Theta_{2}(T_{\mathbf{C}}Z))+2^{6}ch(\Theta_{3}(T_{\mathbf{C}}Z))]\\
&\cdot\varphi(\tau)^{(16)}ch(\mathcal{V}_{i})ch(\mathcal{V}_{j})=e^{\frac{1}{720}(c_{2}(W_{i})+c_{2}(W_{j}))}\widehat{A}(TZ)(ch(\triangle(Z))+128)+q\Big(e^{\frac{1}{720}(c_{2}(W_{i})+c_{2}(W_{j}))}\nonumber\\
&\cdot\widehat{A}(TZ)\big(ch(\triangle(Z))ch(-16+2\widetilde{T_{\mathbf{C}}Z}+W_{i}+W_{j})+ch(-2048+128\widetilde{T_{\mathbf{C}}Z}+128W_{i}+128W_{j}\nonumber\\
&+128\wedge^{2}(\widetilde{T_{\mathbf{C}}Z}))-\frac{1}{30}(c_{2}(W_{i})+c_{2}(W_{j}))(ch(\triangle(Z))+128)\big)\Big)+\cdot\cdot\cdot.\nonumber
\end{align}

$Q_{2}(Z, \tau)$ must be a multiple of
\begin{align}
E_{4}(\tau)^{2}E_{6}(\tau)=1-24q+\cdot\cdot\cdot.
\end{align}

As in the proof of Theorem 3.2, we can get Theorem 3.5.
\end{proof}

\begin{lem}
Note that if $c_{2}(W_{i})+c_{2}(W_{j})=0,$ then we have
\begin{align}
&-\{\widehat{A}(TZ)ch(\triangle(Z))ch(8+2\widetilde{T_{\mathbf{C}}Z}+W_{i}+W_{j}) \}^{(12)}\\
&=\{\widehat{A}(TZ)ch(1024+128\widetilde{T_{\mathbf{C}}Z}+128W_{i}+128W_{j}+128\wedge^{2}(\widetilde{T_{\mathbf{C}}Z})) \}^{(12)}.\nonumber
\end{align}
\end{lem}

\subsection{ Anomaly cancellation formulas for the gauge group $E_{8}\times E_{8}\times E_{8}$ }

Write
\begin{align}
Q_{3}(Z, \tau)=&\int_{Z}e^{\frac{1}{24}E_{2}(\tau)\cdot\frac{1}{30}(c_{2}(W_{i})+c_{2}(W_{j})+c_{2}(W_{r}))}\widehat{A}(TZ)[ch(\triangle(Z))ch(\Theta_{1}(T_{\mathbf{C}}Z))\\
&+2^{6}ch(\Theta_{2}(T_{\mathbf{C}}Z))+2^{6}ch(\Theta_{3}(T_{\mathbf{C}}Z))]\varphi(\tau)^{(24)}ch(\mathcal{V}_{i})ch(\mathcal{V}_{j})ch(\mathcal{V}_{r}).\nonumber
\end{align}

We can now state our main result of this section as follows.
\begin{lem}
$Q_{3}(Z, \tau)$ is a modular form of weight $18$ over $SL_{2}(\mathbf{Z}).$
\end{lem}
\begin{proof}
It is easily seen that
\begin{align}
&Q_{3}(Z, \tau)=\\
&e^{\frac{1}{24}E_{2}(\tau)\cdot\frac{1}{30}(c_{2}(W_{i})+c_{2}(W_{j})+c_{2}(W_{r}))}\prod_{k=1}^{6}\frac{2x_{k}\theta'(0, \tau)}{\theta(x_{k}, \tau)}\Big(\prod_{k=1}^{6}\frac{\theta_{1}(x_{k}, \tau)}{\theta_{1}(0, \tau)}+\prod_{k=1}^{6}\frac{\theta_{2}(x_{k}, \tau)}{\theta_{2}(0, \tau)}+\prod_{k=1}^{6}\frac{\theta_{3}(x_{k}, \tau)}{\theta_{3}(0, \tau)}\Big)\nonumber\\
&\cdot\frac{1}{8}\Big(\prod_{l=1}^{8}\theta_{1}(y_{l}^{i}, \tau)+\prod_{l=1}^{8}\theta_{2}(y_{l}^{i}, \tau)+\prod_{l=1}^{8}\theta_{3}(y_{l}^{i}, \tau)\Big)\Big(\prod_{l=1}^{8}\theta_{1}(y_{l}^{j}, \tau)+\prod_{l=1}^{8}\theta_{2}(y_{l}^{j}, \tau)+\prod_{l=1}^{8}\theta_{3}(y_{l}^{j}, \tau)\Big)\nonumber\\
&\cdot\Big(\prod_{l=1}^{8}\theta_{1}(y_{l}^{r}, \tau)+\prod_{l=1}^{8}\theta_{2}(y_{l}^{r}, \tau)+\prod_{l=1}^{8}\theta_{3}(y_{l}^{r}, \tau)\Big).\nonumber
\end{align}
By (2.25) and (2.26), we have
\begin{align}
&Q_{3}(Z, \tau+1)=Q_{3}(Z, \tau),\\
&Q_{3}(Z, -\frac{1}{\tau})=\tau^{18}Q_{3}(Z, \tau).
\end{align}
Clearly, we can get Lemma 3.7. 
\end{proof}

\begin{thm}
For $12$ dimensional oriented smooth closed manifold, we deduce that
\begin{align}
&\Big\{e^{\frac{1}{720}(c_{2}(W_{i})+c_{2}(W_{j})+c_{2}(W_{r}))}\widehat{A}(TZ)\Big(\frac{1}{30}(c_{2}(W_{i})+c_{2}(W_{j})+c_{2}(W_{r}))ch(\triangle(Z))ch(-504-2\widetilde{T_{\mathbf{C}}Z}\\
&-W_{i}-W_{j}-W_{r})+\frac{1}{30}(c_{2}(W_{i})+c_{2}(W_{j})+c_{2}(W_{r}))ch(-64512-128\widetilde{T_{\mathbf{C}}Z}-128W_{i}-128W_{j}\nonumber\\
&-128W_{r}-128\wedge^{2}(\widetilde{T_{\mathbf{C}}Z}))+\frac{1}{2}\big(-6+\frac{1}{30}(c_{2}(W_{i})+c_{2}(W_{j})+c_{2}(W_{r}))\big)\frac{1}{30}(c_{2}(W_{i})+c_{2}(W_{j})\nonumber\\
&+c_{2}(W_{r}))(ch(\triangle(Z))+128)\Big)\Big\}^{(12)}=\big\{e^{\frac{1}{720}(c_{2}(W_{i})+c_{2}(W_{j})+c_{2}(W_{r}))}\widehat{A}(TZ)\big(ch(\triangle(Z))ch(-73764\nonumber\\
&-504W_{i}-504W_{j}-504W_{r}-\overline{W_{i}}-\overline{W_{j}}-\overline{W_{r}}-1058\widetilde{T_{\mathbf{C}}Z}-\widetilde{T_{\mathbf{C}}Z}\otimes\widetilde{T_{\mathbf{C}}Z}-S^{2}(\widetilde{T_{\mathbf{C}}Z})\nonumber\\
&-\wedge^{2}(\widetilde{T_{\mathbf{C}}Z}))-ch(\triangle(Z))ch(2\widetilde{T_{\mathbf{C}}Z})ch(-24+W_{i}+W_{j}+W_{r})-ch(\triangle(Z))ch(W_{i})ch(W_{j})\nonumber\\
&-ch(\triangle(Z))ch(W_{i})ch(W_{r})-ch(\triangle(Z))ch(W_{j})ch(W_{r})+ch(-9409536-67584\widetilde{T_{\mathbf{C}}Z}\nonumber\\
&-67584W_{i}-67584W_{j}-67584W_{r}-67584\wedge^{2}(\widetilde{T_{\mathbf{C}}Z}))\big)\big\}^{(12)}.\nonumber
\end{align}
\end{thm}
\begin{proof}
We denote $\frac{1}{30}(c_{2}(W_{i})+c_{2}(W_{j})+c_{2}(W_{r}))$ briefly by $A.$
A trivial verification shows that
\begin{align}
&e^{\frac{1}{24}E_{2}(\tau)A}\widehat{A}(TZ)[ch(\triangle(Z))ch(\Theta_{1}(T_{\mathbf{C}}Z))+2^{6}ch(\Theta_{2}(T_{\mathbf{C}}Z))+2^{6}ch(\Theta_{3}(T_{\mathbf{C}}Z))]\\
&\cdot\varphi(\tau)^{(24)}ch(\mathcal{V}_{i})ch(\mathcal{V}_{j})ch(\mathcal{V}_{r})=e^{\frac{1}{24}A}\widehat{A}(TZ)(ch(\triangle(Z))+128)+qe^{\frac{1}{24}A}\widehat{A}(TZ)\big(ch(\triangle(Z))\nonumber\\
&\cdot ch(-24+2\widetilde{T_{\mathbf{C}}Z}+W_{i}+W_{j}+W_{r})+ch(-3072+128\widetilde{T_{\mathbf{C}}Z}+128W_{i}+128W_{j}+128W_{r}\nonumber\\
&+128\wedge^{2}(\widetilde{T_{\mathbf{C}}Z}))-A(ch(\triangle(Z))+128)\big)+q^{2}e^{\frac{1}{24}A}\widehat{A}(TZ)\big(ch(\triangle(Z))ch(252+\overline{W_{i}}+\overline{W_{j}}\nonumber\\
&+\overline{W_{r}}-24W_{i}-24W_{j}-24W_{r})+ch(\triangle(Z))ch(W_{i})ch(W_{j})+ch(\triangle(Z))ch(W_{i})ch(W_{r})\nonumber\\
&+ch(\triangle(Z))ch(W_{j})ch(W_{r})+ch(\triangle(Z))ch(2\widetilde{T_{\mathbf{C}}Z})ch(-24+W_{i}+W_{j}+W_{r})+ch(\triangle(Z))\nonumber\\
&\cdot ch(2\widetilde{T_{\mathbf{C}}Z}+\widetilde{T_{\mathbf{C}}Z}\otimes\widetilde{T_{\mathbf{C}}Z}+S^{2}(\widetilde{T_{\mathbf{C}}Z})+\wedge^{2}(\widetilde{T_{\mathbf{C}}Z}))+Ach(\triangle(Z))ch(24-2\widetilde{T_{\mathbf{C}}Z}-W_{i}\nonumber\\
&-W_{j}-W_{r})+Ach(3072-128\widetilde{T_{\mathbf{C}}Z}-128W_{i}-128W_{j}-128W_{r}-128\wedge^{2}(\widetilde{T_{\mathbf{C}}Z}))+\frac{1}{2}\nonumber\\
&\cdot(-6+A)A(ch(\triangle(Z))+128)\big)+\cdot\cdot\cdot.\nonumber
\end{align}

$Q_{3}(Z, \tau)$ is a modular form of weight $18$ over $SL_{2}(\mathbf{Z}),$ so
\begin{align}
Q_{3}(Z, \tau)=\lambda_{1}E_{4}(\tau)^{3}E_{6}(\tau)+\lambda_{2}E_{6}(\tau)^{3}.
\end{align}

Comparing the coefficients of $1, q$ and $q^{2},$ we can get the Theorem 3.8.
\end{proof}

\begin{lem}
For $c_{2}(W_{i})+c_{2}(W_{j})+c_{2}(W_{r})=0,$ a straightforward calculation shows that
\begin{align}
&\big\{\widehat{A}(TZ)ch(9409536+67584\widetilde{T_{\mathbf{C}}Z}+67584W_{i}+67584W_{j}+67584W_{r}+67584\wedge^{2}(\widetilde{T_{\mathbf{C}}Z}))\big\}^{(12)}\\
&=\big\{\widehat{A}(TZ)\big(ch(\triangle(Z))ch(-73764-504W_{i}-504W_{j}-504W_{r}-\overline{W_{i}}-\overline{W_{j}}-\overline{W_{r}}-1058\widetilde{T_{\mathbf{C}}Z}\nonumber\\
&-\widetilde{T_{\mathbf{C}}Z}\otimes\widetilde{T_{\mathbf{C}}Z}-S^{2}(\widetilde{T_{\mathbf{C}}Z})-\wedge^{2}(\widetilde{T_{\mathbf{C}}Z}))-ch(\triangle(Z))ch(2\widetilde{T_{\mathbf{C}}Z})ch(-24+W_{i}+W_{j}+W_{r})\nonumber\\
&-ch(\triangle(Z))ch(W_{i})ch(W_{j})-ch(\triangle(Z))ch(W_{i})ch(W_{r})-ch(\triangle(Z))ch(W_{j})ch(W_{r})\big)\big\}^{(12)}.\nonumber
\end{align}
\end{lem}

Let $\xi$ be a rank two real oriented Euclidean vector bundle over $Z$ carrying with a Euclidean connection $\nabla^{\xi}.$
Let $c=e(\xi, \nabla^{\xi})=2\pi\sqrt{-1}u$ be the Euler form canonically associated to $\nabla^{\xi}.$

Following \cite{CHZ,HHLZ}, set
\begin{align}
\Theta(T_{\mathbf{C}}Z)=\bigotimes^{\infty}_{m=1}S_{q^{m}}(\widetilde{T_{\mathbf{C}}Z})\otimes\bigotimes_{n=1}^{\infty}\wedge_{q^{n}}(\widetilde{\xi_{\mathbf{C}}})\otimes\bigotimes_{u=1}^{\infty}\wedge_{-q^{u-\frac{1}{2}}}(\widetilde{\xi_{\mathbf{C}}})\otimes\bigotimes_{v=1}^{\infty}\wedge_{q^{v-\frac{1}{2}}}(\widetilde{\xi_{\mathbf{C}}})\in K(Z)[[q]],
\end{align}
where $\xi_{\mathbf{C}}$ is the complexification of $\xi.$
From this $\Theta(T_{\mathbf{C}}Z)$ admits a formal Fourier expansion in $q$ as
\begin{align}
\Theta(T_{\mathbf{C}}Z)=\mathbf{C}+B_{1}q+B_{2}q^{2}+\cdot\cdot\cdot,
\end{align}
where $B_{j}$  are elements in the semi-group formally generated by complex vector bundles over $Z.$
Moreover, they carry canonically induced connections denoted by $\nabla^{B_{j}}.$ 

Set $p_{1}$ denote the first Pontryagin form.
For $1\leq i, j, r\leq2,$ let
\begin{align}
Q(P_{i}, P_{j}, P_{r}, \tau)=&\Big\{ e^{\frac{1}{24}E_{2}(\tau)[p_{1}(TZ)-3c^{2}+\frac{1}{30}(c_{2}(W_{i})+c_{2}(W_{j})+c_{2}(W_{r}))]}\\
&\cdot\widehat{A}(TZ)\cosh\Big(\frac{c}{2}\Big)ch(\Theta(T_{\mathbf{C}}Z))\varphi(\tau)^{(24)}ch(\mathcal{V}_{i})ch(\mathcal{V}_{j})ch(\mathcal{V}_{r})\Big\}^{(12)}.\nonumber
\end{align}

\begin{lem}
$Q(P_{i}, P_{j}, P_{r}, \tau)$ is a modular form of weight $18$ over $SL_{2}(\mathbf{Z}).$
\end{lem}
\begin{proof}
Since
\begin{align}
&\widehat{A}(TZ)ch\Big(\bigotimes^{\infty}_{m=1}S_{q^{m}}(\widetilde{T_{\mathbf{C}}Z})\Big)=\prod_{k=1}^{6}\frac{x_{k}\theta'(0, \tau)}{\theta(x_{k}, \tau)},\\
&\cosh\Big(\frac{c}{2}\Big)ch\Big(\bigotimes_{n=1}^{\infty}\wedge_{q^{n}}(\widetilde{\xi_{\mathbf{C}}})\Big)=\frac{\theta_{1}(u, \tau)}{\theta_{1}(0, \tau)},
\end{align}
\begin{align}
&ch\Big(\bigotimes_{u=1}^{\infty}\wedge_{-q^{u-\frac{1}{2}}}(\widetilde{\xi_{\mathbf{C}}})\Big)=\frac{\theta_{2}(u, \tau)}{\theta_{2}(0, \tau)},\\
&ch\Big(\bigotimes_{v=1}^{\infty}\wedge_{q^{v-\frac{1}{2}}}(\widetilde{\xi_{\mathbf{C}}})\Big)=\frac{\theta_{3}(u, \tau)}{\theta_{3}(0, \tau)},
\end{align}
we can get
\begin{align}
&Q(P_{i}, P_{j}, P_{r}, \tau)=\\
&\bigg\{\frac{1}{8}e^{\frac{1}{24}E_{2}(\tau)[p_{1}(TZ)-3c^{2}+\frac{1}{30}(c_{2}(W_{i})+c_{2}(W_{j})+c_{2}(W_{r}))]}\prod_{k=1}^{6}\frac{x_{k}\theta'(0, \tau)}{\theta(x_{k}, \tau)}\frac{\theta_{1}(u, \tau)}{\theta_{1}(0, \tau)}\frac{\theta_{2}(u, \tau)}{\theta_{2}(0, \tau)}\frac{\theta_{3}(u, \tau)}{\theta_{3}(0, \tau)}\nonumber\\
&\cdot\Big(\prod_{l=1}^{8}\theta_{1}(y_{l}^{i}, \tau)+\prod_{l=1}^{8}\theta_{2}(y_{l}^{i}, \tau)+\prod_{l=1}^{8}\theta_{3}(y_{l}^{i}, \tau)\Big)\Big(\prod_{l=1}^{8}\theta_{1}(y_{l}^{j}, \tau)+\prod_{l=1}^{8}\theta_{2}(y_{l}^{j}, \tau)+\prod_{l=1}^{8}\theta_{3}(y_{l}^{j}, \tau)\Big)\nonumber\\
&\cdot\Big(\prod_{l=1}^{8}\theta_{1}(y_{l}^{r}, \tau)+\prod_{l=1}^{8}\theta_{2}(y_{l}^{r}, \tau)+\prod_{l=1}^{8}\theta_{3}(y_{l}^{r}, \tau)\Big)\bigg\}^{(12)}.\nonumber
\end{align}
In this way, $Q(P_{i}, P_{j}, P_{r}, \tau)$ is a modular form over $SL_{2}(\mathbf{Z})$ with the weight $18.$
\end{proof}

\begin{thm}
For $12$ dimensional oriented smooth closed manifold, we obtain
\begin{align}
&\Big\{e^{\frac{1}{24}(p_{1}(TZ)-3c^{2}+\frac{1}{30}(c_{2}(W_{i})+c_{2}(W_{j})+c_{2}(W_{r})))}\frac{1}{2}(-6+p_{1}(TZ)-3c^{2}+\frac{1}{30}(c_{2}(W_{i})+c_{2}(W_{j})\\
&+c_{2}(W_{r})))(p_{1}(TZ)-3c^{2}+\frac{1}{30}(c_{2}(W_{i})+c_{2}(W_{j})+c_{2}(W_{r})))\widehat{A}(TZ)\cosh\Big(\frac{c}{2}\Big)\nonumber\\
&-e^{\frac{1}{24}(p_{1}(TZ)-3c^{2}+\frac{1}{30}(c_{2}(W_{i})+c_{2}(W_{j})+c_{2}(W_{r})))}(p_{1}(TZ)-3c^{2}+\frac{1}{30}(c_{2}(W_{i})+c_{2}(W_{j})\nonumber\\
&+c_{2}(W_{r})))\widehat{A}(TZ)\cosh\Big(\frac{c}{2}\Big)ch(504+B_{1}+W_{i}+W_{j}+W_{r})\Big\}^{(12)}\nonumber\\
&=\Big\{e^{\frac{1}{24}(p_{1}(TZ)-3c^{2}+\frac{1}{30}(c_{2}(W_{i})+c_{2}(W_{j})+c_{2}(W_{r})))}\widehat{A}(TZ)\cosh\Big(\frac{c}{2}\Big)\big(ch(-73764-504B_{1}\nonumber\\
&-B_{2}-504W_{i}-504W_{j}-504W_{r}-\overline{W_{i}}-\overline{W_{j}}-\overline{W_{r}})-ch(B_{1})ch(W_{i}+W_{j}+W_{r})\nonumber\\
&-ch(W_{i})ch(W_{j})-ch(W_{i})ch(W_{r})-ch(W_{j})ch(W_{r})\big)\Big\}^{(12)},\nonumber
\end{align}
where
\begin{align}
B_{1}&=\widetilde{T_{\mathbf{C}}Z}+\widetilde{\xi_{\mathbf{C}}}-\widetilde{\xi_{\mathbf{C}}}\otimes\widetilde{\xi_{\mathbf{C}}}+2\wedge^{2}(\widetilde{\xi_{\mathbf{C}}}),
\end{align}
\begin{align}
B_{2}&=\widetilde{T_{\mathbf{C}}Z}+S^{2}(\widetilde{T_{\mathbf{C}}Z})+\widetilde{T_{\mathbf{C}}Z}\otimes\widetilde{\xi_{\mathbf{C}}}-\widetilde{T_{\mathbf{C}}Z}\otimes\widetilde{\xi_{\mathbf{C}}}\otimes\widetilde{\xi_{\mathbf{C}}}+2\widetilde{T_{\mathbf{C}}Z}\otimes\wedge^{2}(\widetilde{\xi_{\mathbf{C}}})\\
&+\widetilde{\xi_{\mathbf{C}}}+\wedge^{2}(\widetilde{\xi_{\mathbf{C}}})-\widetilde{\xi_{\mathbf{C}}}\otimes\widetilde{\xi_{\mathbf{C}}}\otimes\widetilde{\xi_{\mathbf{C}}}+2\widetilde{\xi_{\mathbf{C}}}\otimes\wedge^{2}(\widetilde{\xi_{\mathbf{C}}})-2\widetilde{\xi_{\mathbf{C}}}\otimes\wedge^{3}(\widetilde{\xi_{\mathbf{C}}})\nonumber\\
&+\wedge^{2}(\widetilde{\xi_{\mathbf{C}}})\otimes\wedge^{2}(\widetilde{\xi_{\mathbf{C}}})+2\wedge^{4}(\widetilde{\xi_{\mathbf{C}}}).\nonumber
\end{align}
\end{thm}
\begin{proof}
It is easy to obtain
\begin{align}
&e^{\frac{1}{24}E_{2}(\tau)[p_{1}(TZ)-3c^{2}+\frac{1}{30}(c_{2}(W_{i})+c_{2}(W_{j})+c_{2}(W_{r}))]}\widehat{A}(TZ)\cosh\Big(\frac{c}{2}\Big)ch(\Theta(T_{\mathbf{C}}Z))\varphi(\tau)^{(24)}\\
&\cdot ch(\mathcal{V}_{i})ch(\mathcal{V}_{j})ch(\mathcal{V}_{r})=e^{\frac{1}{24}(p_{1}(TZ)-3c^{2}+\frac{1}{30}(c_{2}(W_{i})+c_{2}(W_{j})+c_{2}(W_{r})))}\widehat{A}(TZ)\cosh\Big(\frac{c}{2}\Big)\nonumber\\
&+q\big(-e^{\frac{1}{24}(p_{1}(TZ)-3c^{2}+\frac{1}{30}(c_{2}(W_{i})+c_{2}(W_{j})+c_{2}(W_{r})))}\widehat{A}(TZ)\cosh\Big(\frac{c}{2}\Big)(p_{1}(TZ)-3c^{2}\nonumber\\
&+\frac{1}{30}(c_{2}(W_{i})+c_{2}(W_{j})+c_{2}(W_{r})))+e^{\frac{1}{24}(p_{1}(TZ)-3c^{2}+\frac{1}{30}(c_{2}(W_{i})+c_{2}(W_{j})+c_{2}(W_{r})))}\widehat{A}(TZ)\nonumber\\
&\cdot\cosh\Big(\frac{c}{2}\Big)ch(-24+B_{1}+W_{i}+W_{j}+W_{r})\big)+q^{2}\big(e^{\frac{1}{24}(p_{1}(TZ)-3c^{2}+\frac{1}{30}(c_{2}(W_{i})+c_{2}(W_{j})+c_{2}(W_{r})))}\nonumber\\
&\cdot\frac{1}{2}(-6+p_{1}(TZ)-3c^{2}+\frac{1}{30}(c_{2}(W_{i})+c_{2}(W_{j})+c_{2}(W_{r})))(p_{1}(TZ)-3c^{2}+\frac{1}{30}(c_{2}(W_{i})\nonumber\\
&+c_{2}(W_{j})+c_{2}(W_{r})))\widehat{A}(TZ)\cosh\Big(\frac{c}{2}\Big)-e^{\frac{1}{24}(p_{1}(TZ)-3c^{2}+\frac{1}{30}(c_{2}(W_{i})+c_{2}(W_{j})+c_{2}(W_{r})))}(p_{1}(TZ)\nonumber\\
&-3c^{2}+\frac{1}{30}(c_{2}(W_{i})+c_{2}(W_{j})+c_{2}(W_{r})))\widehat{A}(TZ)\cosh\Big(\frac{c}{2}\Big)ch(-24+B_{1}+W_{i}+W_{j}+W_{r})\nonumber\\
&+e^{\frac{1}{24}(p_{1}(TZ)-3c^{2}+\frac{1}{30}(c_{2}(W_{i})+c_{2}(W_{j})+c_{2}(W_{r})))}\widehat{A}(TZ)\cosh\Big(\frac{c}{2}\Big)ch(252+B_{2}+\overline{W_{i}}+\overline{W_{j}}+\overline{W_{r}}\nonumber\\
&-24B_{1}+W_{i}B_{1}+W_{j}B_{1}+W_{r}B_{1}-24W_{i}-24W_{j}-24W_{r}+W_{i}W_{j}+W_{i}W_{r}+W_{j}W_{r})\big)\nonumber\\
&+\cdot\cdot\cdot.\nonumber
\end{align}

Consequently,
\begin{align}
Q(P_{i}, P_{j}, P_{r}, \tau)=\lambda_{1}E_{4}(\tau)^{3}E_{6}(\tau)+\lambda_{2}E_{6}(\tau)^{3}.
\end{align}

Similar to the proof of Theorem 3.8, we have (3.44).

To find $B_{1}$ and $B_{2},$ we have
\begin{align}
&\bigotimes^{\infty}_{m=1}S_{q^{m}}(\widetilde{T_{\mathbf{C}}Z})=1+q\widetilde{T_{\mathbf{C}}Z}+q^{2}(\widetilde{T_{\mathbf{C}}Z}+S^{2}(\widetilde{T_{\mathbf{C}}Z}))+\cdot\cdot\cdot;\\
&\bigotimes_{n=1}^{\infty}\wedge_{q^{n}}(\widetilde{\xi_{\mathbf{C}}})=1+q\widetilde{\xi_{\mathbf{C}}}+q^{2}(\widetilde{\xi_{\mathbf{C}}}+\wedge^{2}(\widetilde{\xi_{\mathbf{C}}}))+\cdot\cdot\cdot;
\end{align}
\begin{align}
&\bigotimes_{u=1}^{\infty}\wedge_{-q^{u-\frac{1}{2}}}(\widetilde{\xi_{\mathbf{C}}})=1-q^{\frac{1}{2}}\widetilde{\xi_{\mathbf{C}}}+q\wedge^{2}(\widetilde{\xi_{\mathbf{C}}})+q^{\frac{3}{2}}(-\widetilde{\xi_{\mathbf{C}}}-\wedge^{3}(\widetilde{\xi_{\mathbf{C}}}))\\
&+q^{2}(\widetilde{\xi_{\mathbf{C}}}\otimes\widetilde{\xi_{\mathbf{C}}}+\wedge^{4}(\widetilde{\xi_{\mathbf{C}}}))+\cdot\cdot\cdot;\nonumber\\
&\bigotimes_{v=1}^{\infty}\wedge_{q^{v-\frac{1}{2}}}(\widetilde{\xi_{\mathbf{C}}})=1+q^{\frac{1}{2}}\widetilde{\xi_{\mathbf{C}}}+q\wedge^{2}(\widetilde{\xi_{\mathbf{C}}})+q^{\frac{3}{2}}(\widetilde{\xi_{\mathbf{C}}}+\wedge^{3}(\widetilde{\xi_{\mathbf{C}}}))\\
&+q^{2}(\widetilde{\xi_{\mathbf{C}}}\otimes\widetilde{\xi_{\mathbf{C}}}+\wedge^{4}(\widetilde{\xi_{\mathbf{C}}}))+\cdot\cdot\cdot,\nonumber
\end{align}
it follows that
\begin{align}
&\Theta(T_{\mathbf{C}}Z)=1+q(\widetilde{T_{\mathbf{C}}Z}+\widetilde{\xi_{\mathbf{C}}}-\widetilde{\xi_{\mathbf{C}}}\otimes\widetilde{\xi_{\mathbf{C}}}+2\wedge^{2}(\widetilde{\xi_{\mathbf{C}}}))+q^{2}(\widetilde{T_{\mathbf{C}}Z}+S^{2}(\widetilde{T_{\mathbf{C}}Z})+\widetilde{T_{\mathbf{C}}Z}\otimes\widetilde{\xi_{\mathbf{C}}}\\
&-\widetilde{T_{\mathbf{C}}Z}\otimes\widetilde{\xi_{\mathbf{C}}}\otimes\widetilde{\xi_{\mathbf{C}}}+2\widetilde{T_{\mathbf{C}}Z}\otimes\wedge^{2}(\widetilde{\xi_{\mathbf{C}}})+\widetilde{\xi_{\mathbf{C}}}+\wedge^{2}(\widetilde{\xi_{\mathbf{C}}})-\widetilde{\xi_{\mathbf{C}}}\otimes\widetilde{\xi_{\mathbf{C}}}\otimes\widetilde{\xi_{\mathbf{C}}}+2\widetilde{\xi_{\mathbf{C}}}\otimes\wedge^{2}(\widetilde{\xi_{\mathbf{C}}})\nonumber\\
&-2\widetilde{\xi_{\mathbf{C}}}\otimes\wedge^{3}(\widetilde{\xi_{\mathbf{C}}})+\wedge^{2}(\widetilde{\xi_{\mathbf{C}}})\otimes\wedge^{2}(\widetilde{\xi_{\mathbf{C}}})+2\wedge^{4}(\widetilde{\xi_{\mathbf{C}}}))+\cdot\cdot\cdot.\nonumber
\end{align}
Thus
\begin{align}
B_{1}&=\widetilde{T_{\mathbf{C}}Z}+\widetilde{\xi_{\mathbf{C}}}-\widetilde{\xi_{\mathbf{C}}}\otimes\widetilde{\xi_{\mathbf{C}}}+2\wedge^{2}(\widetilde{\xi_{\mathbf{C}}}),\\
B_{2}&=\widetilde{T_{\mathbf{C}}Z}+S^{2}(\widetilde{T_{\mathbf{C}}Z})+\widetilde{T_{\mathbf{C}}Z}\otimes\widetilde{\xi_{\mathbf{C}}}-\widetilde{T_{\mathbf{C}}Z}\otimes\widetilde{\xi_{\mathbf{C}}}\otimes\widetilde{\xi_{\mathbf{C}}}+2\widetilde{T_{\mathbf{C}}Z}\otimes\wedge^{2}(\widetilde{\xi_{\mathbf{C}}})\\
&+\widetilde{\xi_{\mathbf{C}}}+\wedge^{2}(\widetilde{\xi_{\mathbf{C}}})-\widetilde{\xi_{\mathbf{C}}}\otimes\widetilde{\xi_{\mathbf{C}}}\otimes\widetilde{\xi_{\mathbf{C}}}+2\widetilde{\xi_{\mathbf{C}}}\otimes\wedge^{2}(\widetilde{\xi_{\mathbf{C}}})-2\widetilde{\xi_{\mathbf{C}}}\otimes\wedge^{3}(\widetilde{\xi_{\mathbf{C}}})\nonumber\\
&+\wedge^{2}(\widetilde{\xi_{\mathbf{C}}})\otimes\wedge^{2}(\widetilde{\xi_{\mathbf{C}}})+2\wedge^{4}(\widetilde{\xi_{\mathbf{C}}}).\nonumber
\end{align}
\end{proof}

\begin{lem}
Consider $p_{1}(TZ)-3c^{2}+\frac{1}{30}(c_{2}(W_{i})+c_{2}(W_{j})+c_{2}(W_{r}))=0,$ then
\begin{align}
&73764\Big\{\widehat{A}(TZ)\cosh\Big(\frac{c}{2}\Big)\Big\}^{(12)}=\Big\{\widehat{A}(TZ)\cosh\Big(\frac{c}{2}\Big)\big(ch(-504B_{1}-B_{2}-504W_{i}-504W_{j}\\
&-504W_{r}-\overline{W_{i}}-\overline{W_{j}}-\overline{W_{r}})-ch(B_{1})ch(W_{i}+W_{j}+W_{r})-ch(W_{i})ch(W_{j})-ch(W_{i})ch(W_{r})\nonumber\\
&-ch(W_{j})ch(W_{r})\big)\Big\}^{(12)},\nonumber
\end{align}
where
\begin{align}
B_{1}&=\widetilde{T_{\mathbf{C}}Z}+\widetilde{\xi_{\mathbf{C}}}-\widetilde{\xi_{\mathbf{C}}}\otimes\widetilde{\xi_{\mathbf{C}}}+2\wedge^{2}(\widetilde{\xi_{\mathbf{C}}}),\\
B_{2}&=\widetilde{T_{\mathbf{C}}Z}+S^{2}(\widetilde{T_{\mathbf{C}}Z})+\widetilde{T_{\mathbf{C}}Z}\otimes\widetilde{\xi_{\mathbf{C}}}-\widetilde{T_{\mathbf{C}}Z}\otimes\widetilde{\xi_{\mathbf{C}}}\otimes\widetilde{\xi_{\mathbf{C}}}+2\widetilde{T_{\mathbf{C}}Z}\otimes\wedge^{2}(\widetilde{\xi_{\mathbf{C}}})\\
&+\widetilde{\xi_{\mathbf{C}}}+\wedge^{2}(\widetilde{\xi_{\mathbf{C}}})-\widetilde{\xi_{\mathbf{C}}}\otimes\widetilde{\xi_{\mathbf{C}}}\otimes\widetilde{\xi_{\mathbf{C}}}+2\widetilde{\xi_{\mathbf{C}}}\otimes\wedge^{2}(\widetilde{\xi_{\mathbf{C}}})-2\widetilde{\xi_{\mathbf{C}}}\otimes\wedge^{3}(\widetilde{\xi_{\mathbf{C}}})\nonumber\\
&+\wedge^{2}(\widetilde{\xi_{\mathbf{C}}})\otimes\wedge^{2}(\widetilde{\xi_{\mathbf{C}}})+2\wedge^{4}(\widetilde{\xi_{\mathbf{C}}}).\nonumber
\end{align}
\end{lem}

According to the construction method in \cite{L}, we have
\begin{align}
\Phi(T_{\mathbf{C}}Z)&=\bigotimes^{\infty}_{m=1}S_{q^{m}}(\widetilde{T_{\mathbf{C}}Z})\otimes\bigotimes_{n=1}^{\infty}\wedge_{q^{n}}(\widetilde{T_{\mathbf{C}}Z})\otimes\bigotimes_{u=1}^{\infty}\wedge_{-q^{u-\frac{1}{2}}}(\widetilde{T_{\mathbf{C}}Z})\\
&\otimes\bigotimes_{v=1}^{\infty}\wedge_{q^{v-\frac{1}{2}}}(\widetilde{T_{\mathbf{C}}Z})\in K(Z)[[q]].\nonumber
\end{align}
Perform the formal Fourier expansion in $q$ as
\begin{align}
\Phi(T_{\mathbf{C}}Z)=\mathbf{C}+D_{1}q+D_{2}q^{2}+\cdot\cdot\cdot,
\end{align}
where $D_{j}$  are elements in the semi-group formally generated by complex vector bundles over $Z.$

Choose
\begin{align}
Q_{L}(P_{i}, P_{j}, P_{r}, \tau)&=\Big\{e^{\frac{1}{24}E_{2}(\tau)[-2p_{1}(TZ)+\frac{1}{30}(c_{2}(W_{i})+c_{2}(W_{j})+c_{2}(W_{r}))]}\widehat{L}(TZ)ch(\Phi(T_{\mathbf{C}}Z))\\
&\cdot\varphi(\tau)^{(24)}ch(\mathcal{V}_{i})ch(\mathcal{V}_{j})ch(\mathcal{V}_{r})\Big\}^{(12)}.\nonumber
\end{align}

\begin{lem}
$Q_{L}(P_{i}, P_{j}, P_{r}, \tau)$ is a modular form of weight $18$ over $SL_{2}(\mathbf{Z}).$
\end{lem}
\begin{proof}
We check at once that
\begin{align}
&\widehat{L}(TZ)ch\Big(\bigotimes^{\infty}_{m=1}S_{q^{m}}(\widetilde{T_{\mathbf{C}}Z})\otimes\bigotimes_{n=1}^{\infty}\wedge_{q^{n}}(\widetilde{T_{\mathbf{C}}Z})\Big)=2^{6}\prod_{k=1}^{6}\frac{x_{k}\theta'(0, \tau)}{\theta(x_{k}, \tau)}\frac{\theta_{1}(x_{k}, \tau)}{\theta_{1}(0, \tau)},\\
&ch\Big(\bigotimes_{u=1}^{\infty}\wedge_{-q^{u-\frac{1}{2}}}(\widetilde{T_{\mathbf{C}}Z})\Big)=\prod_{k=1}^{6}\frac{\theta_{2}(x_{k}, \tau)}{\theta_{2}(0, \tau)},\\
&ch\Big(\bigotimes_{v=1}^{\infty}\wedge_{q^{v-\frac{1}{2}}}(\widetilde{T_{\mathbf{C}}Z})\Big)=\prod_{k=1}^{6}\frac{\theta_{3}(x_{k}, \tau)}{\theta_{3}(0, \tau)},
\end{align}
then
\begin{align}
&Q_{L}(P_{i}, P_{j}, P_{r}, \tau)=\\
&2^{6}\bigg\{\frac{1}{8}e^{\frac{1}{24}E_{2}(\tau)[-2p_{1}(TZ)+\frac{1}{30}(c_{2}(W_{i})+c_{2}(W_{j})+c_{2}(W_{r}))]}\prod_{k=1}^{6}\frac{x_{k}\theta'(0, \tau)}{\theta(x_{k}, \tau)}\frac{\theta_{1}(x_{k}, \tau)}{\theta_{1}(0, \tau)}\frac{\theta_{2}(x_{k}, \tau)}{\theta_{2}(0, \tau)}\frac{\theta_{3}(x_{k}, \tau)}{\theta_{3}(0, \tau)}\nonumber\\
&\cdot\Big(\prod_{l=1}^{8}\theta_{1}(y_{l}^{i}, \tau)+\prod_{l=1}^{8}\theta_{2}(y_{l}^{i}, \tau)+\prod_{l=1}^{8}\theta_{3}(y_{l}^{i}, \tau)\Big)\Big(\prod_{l=1}^{8}\theta_{1}(y_{l}^{j}, \tau)+\prod_{l=1}^{8}\theta_{2}(y_{l}^{j}, \tau)+\prod_{l=1}^{8}\theta_{3}(y_{l}^{j}, \tau)\Big)\nonumber\\
&\cdot\Big(\prod_{l=1}^{8}\theta_{1}(y_{l}^{r}, \tau)+\prod_{l=1}^{8}\theta_{2}(y_{l}^{r}, \tau)+\prod_{l=1}^{8}\theta_{3}(y_{l}^{r}, \tau)\Big)\bigg\}^{(12)}.\nonumber
\end{align}
Obviously,
\begin{align}
&Q_{L}(P_{i}, P_{j}, P_{r}, \tau+1)=Q_{L}(P_{i}, P_{j}, P_{r}, \tau),\\
&Q_{L}(P_{i}, P_{j}, P_{r}, -\frac{1}{\tau})=\tau^{18}Q_{L}(P_{i}, P_{j}, P_{r}, \tau),
\end{align}
we thus get
$Q_{L}(P_{i}, P_{j}, P_{r}, \tau)$ is a modular form of weight $18$ over $SL_{2}(\mathbf{Z}).$
\end{proof}

\begin{thm}
For $12$ dimensional oriented smooth closed manifold, we see that
\begin{align}
&\Big\{e^{\frac{1}{24}(-2p_{1}(TZ)+\frac{1}{30}(c_{2}(W_{i})+c_{2}(W_{j})+c_{2}(W_{r})))}\frac{1}{2}(-6-2p_{1}(TZ)+\frac{1}{30}(c_{2}(W_{i})+c_{2}(W_{j})+c_{2}(W_{r})))\\
&\cdot(-2p_{1}(TZ)+\frac{1}{30}(c_{2}(W_{i})+c_{2}(W_{j})+c_{2}(W_{r})))\widehat{L}(TZ)-e^{\frac{1}{24}(-2p_{1}(TZ)+\frac{1}{30}(c_{2}(W_{i})+c_{2}(W_{j})+c_{2}(W_{r})))}\nonumber\\
&\cdot(-2p_{1}(TZ)+\frac{1}{30}(c_{2}(W_{i})+c_{2}(W_{j})+c_{2}(W_{r})))\widehat{L}(TZ)ch(504+D_{1}+W_{i}+W_{j}+W_{r})\Big\}^{(12)}\nonumber\\
&=\Big\{e^{\frac{1}{24}(-2p_{1}(TZ)+\frac{1}{30}(c_{2}(W_{i})+c_{2}(W_{j})+c_{2}(W_{r})))}\widehat{L}(TZ)\big(ch(-73764-504D_{1}-D_{2}-504W_{i}\nonumber\\
&-504W_{j}-504W_{r}-\overline{W_{i}}-\overline{W_{j}}-\overline{W_{r}})-ch(D_{1})ch(W_{i}+W_{j}+W_{r})-ch(W_{i})ch(W_{j})\nonumber\\
&-ch(W_{i})ch(W_{r})-ch(W_{j})ch(W_{r})\big)\Big\}^{(12)},\nonumber
\end{align}
where
\begin{align}
D_{1}&=2\widetilde{T_{\mathbf{C}}Z}-\widetilde{T_{\mathbf{C}}Z}\otimes\widetilde{T_{\mathbf{C}}Z}+2\wedge^{2}(\widetilde{T_{\mathbf{C}}Z}),\\
D_{2}&=2\widetilde{T_{\mathbf{C}}Z}+S^{2}(\widetilde{T_{\mathbf{C}}Z})+\widetilde{T_{\mathbf{C}}Z}\otimes\widetilde{T_{\mathbf{C}}Z}-2\widetilde{T_{\mathbf{C}}Z}\otimes\widetilde{T_{\mathbf{C}}Z}\otimes\widetilde{T_{\mathbf{C}}Z}+4\widetilde{T_{\mathbf{C}}Z}\otimes\wedge^{2}(\widetilde{T_{\mathbf{C}}Z})\\
&+\wedge^{2}(\widetilde{T_{\mathbf{C}}Z})+\wedge^{2}(\widetilde{T_{\mathbf{C}}Z})\otimes\wedge^{2}(\widetilde{T_{\mathbf{C}}Z})-2\widetilde{T_{\mathbf{C}}Z}\otimes\wedge^{3}(\widetilde{T_{\mathbf{C}}Z})+2\wedge^{4}(\widetilde{T_{\mathbf{C}}Z}).\nonumber
\end{align}
\end{thm}
\begin{proof}
It is evident that
\begin{align}
&e^{\frac{1}{24}E_{2}(\tau)[-2p_{1}(TZ)+\frac{1}{30}(c_{2}(W_{i})+c_{2}(W_{j})+c_{2}(W_{r}))]}\widehat{L}(TZ)ch(\Phi(T_{\mathbf{C}}Z))\varphi(\tau)^{(24)}ch(\mathcal{V}_{i})ch(\mathcal{V}_{j})ch(\mathcal{V}_{r})\\
&=e^{\frac{1}{24}(-2p_{1}(TZ)+\frac{1}{30}(c_{2}(W_{i})+c_{2}(W_{j})+c_{2}(W_{r})))}\widehat{L}(TZ)+q\big(-e^{\frac{1}{24}(-2p_{1}(TZ)+\frac{1}{30}(c_{2}(W_{i})+c_{2}(W_{j})+c_{2}(W_{r})))}\nonumber\\
&\cdot\widehat{L}(TZ)(-2p_{1}(TZ)+\frac{1}{30}(c_{2}(W_{i})+c_{2}(W_{j})+c_{2}(W_{r})))+e^{\frac{1}{24}(-2p_{1}(TZ)+\frac{1}{30}(c_{2}(W_{i})+c_{2}(W_{j})+c_{2}(W_{r})))}\nonumber\\
&\cdot\widehat{L}(TZ)ch(-24+D_{1}+W_{i}+W_{j}+W_{r})\big)+q^{2}\big(e^{\frac{1}{24}(-2p_{1}(TZ)+\frac{1}{30}(c_{2}(W_{i})+c_{2}(W_{j})+c_{2}(W_{r})))}\nonumber
\end{align}
\begin{align}
&\cdot\frac{1}{2}(-6-2p_{1}(TZ)+\frac{1}{30}(c_{2}(W_{i})+c_{2}(W_{j})+c_{2}(W_{r})))(-2p_{1}(TZ)+\frac{1}{30}(c_{2}(W_{i})+c_{2}(W_{j})\nonumber\\
&+c_{2}(W_{r})))\widehat{L}(TZ)-e^{\frac{1}{24}(-2p_{1}(TZ)+\frac{1}{30}(c_{2}(W_{i})+c_{2}(W_{j})+c_{2}(W_{r})))}(-2p_{1}(TZ)+\frac{1}{30}(c_{2}(W_{i})+c_{2}(W_{j})\nonumber\\
&+c_{2}(W_{r})))\widehat{L}(TZ)ch(-24+D_{1}+W_{i}+W_{j}+W_{r})+e^{\frac{1}{24}(-2p_{1}(TZ)+\frac{1}{30}(c_{2}(W_{i})+c_{2}(W_{j})+c_{2}(W_{r})))}\nonumber\\
&\cdot\widehat{L}(TZ)ch(252+D_{2}+\overline{W_{i}}+\overline{W_{j}}+\overline{W_{r}}-24D_{1}+W_{i}D_{1}+W_{j}D_{1}+W_{r}D_{1}-24W_{i}-24W_{j}\nonumber\\
&-24W_{r}+W_{i}W_{j}+W_{i}W_{r}+W_{j}W_{r})\big)+\cdot\cdot\cdot.\nonumber
\end{align}

Similarly, we have
\begin{align}
Q_{L}(P_{i}, P_{j}, P_{r}, \tau)=\lambda_{1}E_{4}(\tau)^{3}E_{6}(\tau)+\lambda_{2}E_{6}(\tau)^{3}.
\end{align}
We therefore get (3.68).

What is left is to show that $D_{1}$ and $D_{2}.$
\begin{align}
&\bigotimes_{n=1}^{\infty}\wedge_{q^{n}}(\widetilde{T_{\mathbf{C}}Z})=1+q\widetilde{T_{\mathbf{C}}Z}+q^{2}(\widetilde{T_{\mathbf{C}}Z}+\wedge^{2}(\widetilde{T_{\mathbf{C}}Z}))+\cdot\cdot\cdot,\\
&\bigotimes_{u=1}^{\infty}\wedge_{-q^{u-\frac{1}{2}}}(\widetilde{T_{\mathbf{C}}Z})=1-q^{\frac{1}{2}}\widetilde{T_{\mathbf{C}}Z}+q\wedge^{2}(\widetilde{T_{\mathbf{C}}Z})+q^{\frac{3}{2}}(-\widetilde{T_{\mathbf{C}}Z}-\wedge^{3}(\widetilde{T_{\mathbf{C}}Z}))\\
&+q^{2}(\widetilde{T_{\mathbf{C}}Z}\otimes\widetilde{T_{\mathbf{C}}Z}+\wedge^{4}(\widetilde{T_{\mathbf{C}}Z}))+\cdot\cdot\cdot,\nonumber\\
&\bigotimes_{v=1}^{\infty}\wedge_{q^{v-\frac{1}{2}}}(\widetilde{T_{\mathbf{C}}Z})=1+q^{\frac{1}{2}}\widetilde{T_{\mathbf{C}}Z}+q\wedge^{2}(\widetilde{T_{\mathbf{C}}Z})+q^{\frac{3}{2}}(\widetilde{T_{\mathbf{C}}Z}+\wedge^{3}(\widetilde{T_{\mathbf{C}}Z}))\\
&+q^{2}(\widetilde{T_{\mathbf{C}}Z}\otimes\widetilde{T_{\mathbf{C}}Z}+\wedge^{4}(\widetilde{T_{\mathbf{C}}Z}))+\cdot\cdot\cdot,\nonumber
\end{align}
it follows immediately that
\begin{align}
&\Theta(T_{\mathbf{C}}Z)=1+q(2\widetilde{T_{\mathbf{C}}Z}-\widetilde{T_{\mathbf{C}}Z}\otimes\widetilde{T_{\mathbf{C}}Z}+2\wedge^{2}(\widetilde{T_{\mathbf{C}}Z}))+q^{2}(2\widetilde{T_{\mathbf{C}}Z}+S^{2}(\widetilde{T_{\mathbf{C}}Z})\\
&+\widetilde{T_{\mathbf{C}}Z}\otimes\widetilde{T_{\mathbf{C}}Z}-2\widetilde{T_{\mathbf{C}}Z}\otimes\widetilde{T_{\mathbf{C}}Z}\otimes\widetilde{T_{\mathbf{C}}Z}+4\widetilde{T_{\mathbf{C}}Z}\otimes\wedge^{2}(\widetilde{T_{\mathbf{C}}Z})+\wedge^{2}(\widetilde{T_{\mathbf{C}}Z})\nonumber\\
&+\wedge^{2}(\widetilde{T_{\mathbf{C}}Z})\otimes\wedge^{2}(\widetilde{T_{\mathbf{C}}Z})-2\widetilde{T_{\mathbf{C}}Z}\otimes\wedge^{3}(\widetilde{T_{\mathbf{C}}Z})+2\wedge^{4}(\widetilde{T_{\mathbf{C}}Z}))+\cdot\cdot\cdot.\nonumber
\end{align}
For this reason
\begin{align}
D_{1}&=2\widetilde{T_{\mathbf{C}}Z}-\widetilde{T_{\mathbf{C}}Z}\otimes\widetilde{T_{\mathbf{C}}Z}+2\wedge^{2}(\widetilde{T_{\mathbf{C}}Z}),\\
D_{2}&=2\widetilde{T_{\mathbf{C}}Z}+S^{2}(\widetilde{T_{\mathbf{C}}Z})+\widetilde{T_{\mathbf{C}}Z}\otimes\widetilde{T_{\mathbf{C}}Z}-2\widetilde{T_{\mathbf{C}}Z}\otimes\widetilde{T_{\mathbf{C}}Z}\otimes\widetilde{T_{\mathbf{C}}Z}+4\widetilde{T_{\mathbf{C}}Z}\otimes\wedge^{2}(\widetilde{T_{\mathbf{C}}Z})\\
&+\wedge^{2}(\widetilde{T_{\mathbf{C}}Z})+\wedge^{2}(\widetilde{T_{\mathbf{C}}Z})\otimes\wedge^{2}(\widetilde{T_{\mathbf{C}}Z})-2\widetilde{T_{\mathbf{C}}Z}\otimes\wedge^{3}(\widetilde{T_{\mathbf{C}}Z})+2\wedge^{4}(\widetilde{T_{\mathbf{C}}Z}).\nonumber
\end{align}
\end{proof}

\begin{lem}
When $-2p_{1}(TZ)+\frac{1}{30}(c_{2}(W_{i})+c_{2}(W_{j})+c_{2}(W_{r}))=0,$ we get
\begin{align}
&73764\Big\{\widehat{L}(TZ)\Big\}^{(12)}=\Big\{\widehat{L}(TZ)\big(ch(-504D_{1}-D_{2}-504W_{i}-504W_{j}-504W_{r}-\overline{W_{i}}-\overline{W_{j}}\\
&-\overline{W_{r}})-ch(D_{1})ch(W_{i}+W_{j}+W_{r})-ch(W_{i})ch(W_{j})-ch(W_{i})ch(W_{r})-ch(W_{j})ch(W_{r})\big)\Big\}^{(12)},\nonumber
\end{align}
where
\begin{align}
D_{1}&=2\widetilde{T_{\mathbf{C}}Z}-\widetilde{T_{\mathbf{C}}Z}\otimes\widetilde{T_{\mathbf{C}}Z}+2\wedge^{2}(\widetilde{T_{\mathbf{C}}Z}),\\
D_{2}&=2\widetilde{T_{\mathbf{C}}Z}+S^{2}(\widetilde{T_{\mathbf{C}}Z})+\widetilde{T_{\mathbf{C}}Z}\otimes\widetilde{T_{\mathbf{C}}Z}-2\widetilde{T_{\mathbf{C}}Z}\otimes\widetilde{T_{\mathbf{C}}Z}\otimes\widetilde{T_{\mathbf{C}}Z}+4\widetilde{T_{\mathbf{C}}Z}\otimes\wedge^{2}(\widetilde{T_{\mathbf{C}}Z})\\
&+\wedge^{2}(\widetilde{T_{\mathbf{C}}Z})+\wedge^{2}(\widetilde{T_{\mathbf{C}}Z})\otimes\wedge^{2}(\widetilde{T_{\mathbf{C}}Z})-2\widetilde{T_{\mathbf{C}}Z}\otimes\wedge^{3}(\widetilde{T_{\mathbf{C}}Z})+2\wedge^{4}(\widetilde{T_{\mathbf{C}}Z}).\nonumber
\end{align}
\end{lem}

\section{ Conclusion }

To enrich the results of anomaly cancellation formulas, we give more anomaly cancellation formulas for the gauge groups $E_{8},$ $E_{8}\times E_{8}$ and $E_{8}\times E_{8}\times E_{8}.$

To begin with, we get $SL_{2}(\mathbf{Z})$ modular forms of weight $10$ and $14$ for $12$ dimensional oriented smooth closed manifold.
Thus we obtain some anomaly cancellation formulas for $E_{8},$ $E_{8}\times E_{8}.$

For another thing, we define $Q(P_{i}, P_{j}, P_{r}, \tau)$ and $Q_{L}(P_{i}, P_{j}, P_{r}, \tau).$
We conclude that $Q(P_{i}, P_{j}, P_{r}, \tau)$ is a modular form over $SL_{2}(\mathbf{Z})$ with the weight $18$ and $Q_{L}(P_{i}, P_{j}, P_{r}, \tau)$ is a modular form over $SL_{2}(\mathbf{Z})$ with the weight $18.$
Therefore we deduce that some anomaly cancellation formulas for $E_{8}\times E_{8}\times E_{8}.$

\vskip 1 true cm

\section{ Acknowledgements }

The author was supported in part by National Natural Science Foundation of China (NSFC) No.11771070. 
The author thanks the referee for his (or her) careful reading and helpful comments.

\vskip 1 true cm

%-----------------------------------------------------------------------------
%-----------------------------------------------------------------------------

\bigskip
\bigskip

\noindent {\footnotesize {\it S. Liu} \\
{School of Mathematics and Statistics, Northeast Normal University, Changchun 130024, China}\\
{Email: liusy719@nenu.edu.cn}

\noindent {\footnotesize {\it Y. Wang} \\
{School of Mathematics and Statistics, Northeast Normal University, Changchun 130024, China}\\
{Email: wangy581@nenu.edu.cn}

\noindent {\footnotesize {\it Y. Yang} \\
{School of Mathematics and Statistics, Northeast Normal University, Changchun 130024, China}\\
{Email: yangyc580@nenu.edu.cn}

\end{document}